\documentclass[12pt]{amsart}

\usepackage{geometry}
\usepackage[utf8]{inputenc}
\usepackage[english]{babel}
\usepackage{amssymb}
\usepackage{amsmath}
\usepackage{amsthm}
\usepackage{amscd}
\usepackage{enumitem}
\usepackage{graphicx}
\usepackage{tikz}

\newcommand{\intd}{\text{d}}

\theoremstyle{plain}
\newtheorem{theorem}{Theorem}[section]
\newtheorem{lemma}[theorem]{Lemma}

\theoremstyle{definition}
\newtheorem{definition}[theorem]{Definition}
\newtheorem{example}[theorem]{Example}
\newtheorem{question}[theorem]{Question}
\newtheorem{remark}[theorem]{Remark}

\title{A note on the Wiman-Valiron inequality}
\author{Karl-G. Grosse-Erdmann}
\address{Karl-G. Grosse-Erdmann, 
D\'epartement de Math\'ematique, Universit\'e de Mons, 20 Place du Parc, 7000 Mons, Belgium}
\email{kg.grosse-erdmann@umons.ac.be}

\thanks{The author is supported by the Fonds de la Recherche Scientifique - FNRS under Grant n\textsuperscript{o} CDR J.0078.21. I am very grateful to Kevin Agneessens for many stimulating discussions.}

\begin{document}

\begin{abstract}
The Wiman-Valiron inequality relates the maximum modulus of an analytic function to its Taylor coefficients via the maximum term. After a short overview of the known results, we obtain a general version of this inequality that seems to have been overlooked in the literature so far. We end the paper with an open problem.
\end{abstract}

\maketitle

\section{The Wiman-Valiron inequality}\label{s-wv}

For an analytic function $f(z)=\sum_{n=0}^\infty a_n z^n$ on an open disk $\{z\in\mathbb{C} : |z|<R\}$, $0<R\leq \infty$, in the complex plane, its maximum modulus and its maximum term
are defined by
\[
M_f(r) = \max_{|z|=r}|f(z)| \ \text{ and } \ \mu_f(r) = \max_{n\geq 0}|a_n|r^n, \ 0\leq r<R,
\]
respectively. By the Cauchy estimates one has that
\[
\mu_f(r) \leq M_f(r), \ 0\leq r<R.
\]
The theory initiated by Wiman and Valiron at the beginning of the last century seeks, among other things, to obtain an upper estimate of $M_f$ via $\mu_f$. For an introduction to the Wiman-Valiron theory for entire functions we refer to Hayman \cite{Hay74}, \cite[Section 6.5]{Hay89}.

\subsection{Wiman-Valiron theory on the plane} The classical Wiman-Valiron inequality says that if $f$ is a non-constant entire function then, for any $\delta>0$, there is a (measurable) set $E\subset [0,\infty)$ of finite logarithmic measure and some $C>0$ such that
\begin{equation}\label{eq-wv}
M_f(r)\leq C \mu_f(r)\big(\log \mu_f(r)\big)^{\frac{1}{2}+\delta}, \ r\notin E.
\end{equation}
Here, $E$ is said to be of finite logarithmic measure if
\[
\int_{E\cap [1,\infty)} \frac{1}{r}\intd r<\infty.
\]
More precisely, Wiman \cite{Wim14} had shown that inequality \eqref{eq-wv} holds for an unbounded sequence of radii $(r_n)_n$, while Valiron \cite{Val18}, \cite{Val20}, \cite[p.\ 106]{Val23} obtained the inequality outside the stated exceptional set $E$.

In 1962, Rosenbloom \cite{Ros62} proposed a new proof that uses, in a clever way, a probabilistic argument. In this way he arrived at the following stronger result, where he considers positive increasing functions $\psi_1, \psi_2$ on some interval $[a,\infty)$, $a\geq 0$, so that, for some $b\geq a$,
\[
\int_b^\infty\frac{1}{\psi_k(y)}\intd y<\infty, \ k=1,2,
\]
and so that $\varphi(y):=\psi_2(\psi_1(y))$ is defined for $y\geq b$.

\begin{theorem}[Rosenbloom]\label{t-rosen}
Let $f$ be a non-constant entire function. Then there is a set $E\subset [0,\infty)$ of finite logarithmic measure and some $C>0$ such that
\[
M_f(r)\leq C \mu_f(r)\sqrt{\varphi\big(\log M_f(r)\big)}, \ r\notin E.
\]
\end{theorem}

If one takes $\psi_1(y)=\psi_2(y)=y^{1+\delta}$, $\delta >0$, one arrives at the classical Wiman-Valiron inequality as a special case; see Remark \ref{r-ineq} below.

For $\psi_1(y)=\psi_2(y)=y(\log y)^{1+\delta}$, $\delta >0$, one obtains the stronger inequality
\begin{equation}\label{eq-wvB}
M_f(r)\leq C \mu_f(r)\big(\log \mu_f(r)\big)^{\frac{1}{2}}\big(\log\log \mu_f(r)\big)^{1+\delta}, \ r\notin E.
\end{equation}

More generally, for $\psi_1(y)=\psi_2(y)=y\log y\log_2y\cdots(\log_{n-1} y)^{1+\delta}$, $n\geq 2$, $\delta >0$, Rosenbloom deduced that
\begin{equation}\label{eq-wvC}
M_f(r)\leq C \mu_f(r)\big(\log \mu_f(r)\big)^{\frac{1}{2}}\log_2 \mu_f(r)\cdots\log_{n-1} \mu_f(r)\big(\log_n \mu_f(r)\big)^{1+\delta}, \ r\notin E,
\end{equation}
see also \cite[Theorem 6]{Hay74}, \cite[Theorem 6.23]{Hay89}; here, $\log_n$ denotes the $n$-fold iterated logarithm. 

Hayman \cite[p.\ 334]{Hay74} notes that inequality \eqref{eq-wvC} is optimal in the sense that $\delta$ cannot be replaced by 0.

\begin{remark}\label{r-ineq}
The derivation of inequalities like \eqref{eq-wv}, \eqref{eq-wvB} and \eqref{eq-wvC} is elementary and tedious, but maybe not completely obvious. Let us therefore perform the steps once, taking \eqref{eq-wvB} for an illustration. If $\psi_1(y)=\psi_2(y)=y(\log y)^{1+\delta}$, $\delta>0$, then we have, for large $y$,
\begin{align}
\varphi(y)=\psi_2(\psi_1(y))&=y(\log y)^{1+\delta}\big(\log(y(\log y)^{1+\delta})\big)^{1+\delta}\notag\\
&\leq y(\log y)^{1+\delta}(\log(y^2))^{1+\delta}= 2^{1+\delta} y(\log y)^{2+2\delta}\label{WV1}\\
&\leq y^2.\label{WV2}
\end{align}
Thus, by Theorem \ref{t-rosen} and \eqref{WV2}, there is a set $E\subset [0,\infty)$ of finite logarithmic measure and some $C>0$ such that, for $r\notin E$ large, 
\[
M_f(r)\leq C \mu_f(r)\sqrt{\varphi\big(\log M_f(r)\big)} \leq C \mu_f(r)\log M_f(r) \leq \mu_f(r)  M_f(r)^{\frac{1}{2}}
\]
and hence
\[
M_f(r)\leq \mu_f(r)^2.
\]
We deduce, again by Theorem \ref{t-rosen} and by \eqref{WV1}, that there is a possibly larger set $E'\subset [0,\infty)$ of finite logarithmic measure and constants $C, C'>0$ such that, for $r\notin E'$ large,
\begin{align*}
M_f(r)&\leq C \mu_f(r)\sqrt{\varphi\big(2\log \mu_f(r)\big)} \leq C' \mu_f(r)\sqrt{\log \mu_f(r)(\log\log \mu_f(r))^{2+2\delta}}\\&=C' \mu_f(r)\big(\log \mu_f(r)\big)^{\frac{1}{2}}\big(\log\log \mu_f(r)\big)^{1+\delta}. 
\end{align*}

\end{remark}

\subsection{Wiman-Valiron theory on the unit disk}
Shortly after the work of Rosenbloom, K\"ov\'ari \cite{Kov66} realised that Rosenbloom's approach can also be used to obtain a Wiman-Valiron inequality on the unit disk. However, a problem arises due to the presence of an additional factor $\frac{1}{1-r}$. In order to obtain an analogue of Rosenbloom's inequality, K\"ov\'ari was forced to allow exceptional sets $E\subset [0,1)$ of logarithmic density zero, which demands that
\[
\lim_{r\to 1}\frac{1}{\log\frac{1}{1-r}}\int_{E\cap[0,r)} \frac{1}{1-t}\intd t=0.
\]
This is more generous than demanding that $E$ be of finite logarithmic measure, which in the unit interval means that
\[
\int_{E} \frac{1}{1-r}\intd r<\infty.
\]
He also had to be more demanding on the positive increasing functions $\psi_1, \psi_2$ on $[a,\infty)$, $a\geq 0$: there is to be some $b\geq a$ so that
\[
\int_b^\infty\frac{1}{\psi_k(y)}\intd y<\infty, \text{ and }  1\leq \frac{\psi_k(y)}{y} \nearrow \infty \text{ as } b\leq y \nearrow \infty 
\]
for $k=1,2$; that
\[
\psi_2(y_1y_2) \leq K\big(y_2\psi_2(y_1)+ y_1\psi_2(y_2)\big)
\]
for some constant $K>0$ and all $y_1, y_2\geq b$; and that $\varphi(y):=\psi_2(\psi_1(y))$ is defined for $y\geq b$. K\"ov\'ari's main result is then as follows.

\begin{theorem}[K\"ov\'ari]\label{t-kov}
Let $f$ be an unbounded analytic function on the unit disk. Then there is a set $E\subset [0,1)$ of logarithmic density zero and some $C>0$ such that
\[
M_f(r)\leq C \frac{\mu_f(r)}{1-r}\sqrt{\varphi\big(\log M_f(r)\big)}, \ r\notin E.
\]
\end{theorem}

Since the functions $\psi_1(y)=y^{1+\delta}$, $\delta >0$, and $\psi_2(y)=y(\log y)^2$ satisfy the hypotheses of K\"ov\'ari's theorem, one obtains the inequality
\begin{equation}\label{eq-wvD}
M_f(r)\leq C \frac{\mu_f(r)}{1-r}\Big(\log \frac{\mu_f(r)}{1-r}\Big)^{\frac{1}{2}+\delta}, \ r\notin E.
\end{equation}

In the same way, for $\psi_1(y)=\psi_2(y)=y\log y\log_2y\cdots(\log_{n-1} y)^{1+\delta}$, $n\geq 2$, $\delta >0$, one deduces the stronger inequality
\begin{equation}\label{eq-wvD2}
M_f(r)\leq C \frac{\mu_f(r)}{1-r}\Big(\log \frac{\mu_f(r)}{1-r}\Big)^{\frac{1}{2}}\log_2 \frac{\mu_f(r)}{1-r}\cdots\log_{n-1} \frac{\mu_f(r)}{1-r}\Big(\log_n \frac{\mu_f(r)}{1-r}\Big)^{1+\delta}, \ r\notin E.
\end{equation}

In 1980, Sule\u{\i}manov \cite{Sul80} states a result that holds with exceptional sets of finite logarithmic measure. For this he considers a positive differentiable function $\varphi$ on $[a,\infty)$, $a>0$, such that
\[
\int_a^\infty \Big(\int_a^y \varphi(t)\intd t\Big)^{-\frac{1}{2}} \intd y<\infty \ \text{ and } y\varphi'(y)\leq K\varphi(y), \ y\geq a,
\]
for some constant $K>0$. He does not give a proof, but it appears that he was also building on ideas of Rosenbloom.

\begin{theorem}[Sule\u{\i}manov]\label{t-sou}
Let $f$ be an unbounded analytic function on the unit disk. Then there is a set $E\subset [0,1)$ of finite logarithmic measure and some $C>0$ such that
\[
M_f(r)\leq C \frac{\mu_f(r)}{(1-r)^{\frac{3}{4}}}\sqrt{\varphi\Big(\frac{1}{(1-r)^{\frac{1}{2}}}\log M_f(r)\Big)}, \ r\notin E.
\]
\end{theorem}

In the special case when $\varphi(y)=y^{1+\delta}$, $\delta>0$, he obtains the inequality
\begin{equation}\label{eq-sul}
M_f(r)\leq C \frac{\mu_f(r)}{(1-r)^{1+\delta}}\Big(\log \frac{\mu_f(r)}{1-r}\Big)^{\frac{1}{2}+\delta}, \ r\notin E.
\end{equation}

He mentions that one may apply the result also with $\varphi(y)=y(\log y)^2(\log_2y)^2\cdots(\log_{n-1} y)^{2+\delta}$, $n\geq 2$, $\delta >0$. For such a choice one is led to the inequality
\begin{equation}\label{eq-sul2}
M_f(r)\leq C \frac{\mu_f(r)}{1-r}\Big(\log \frac{1}{1-r}\Big)^{1+\delta}\Big(\log \frac{\mu_f(r)}{1-r}\Big)^{\frac{1}{2}}\log_2\frac{\mu_f(r)}{1-r}\cdots\log_{n-1} \frac{\mu_f(r)}{1-r}\Big(\log_n \frac{\mu_f(r)}{1-r}\Big)^{1+\delta}
\end{equation}
for $r\notin E$.

More recently, Fenton and Strumia \cite{FeSt09} returned to the classical approach of Wiman and Valiron. As for their exceptional set $E$, they want it to be of final density zero, which means that
\[
\lim_{r\to 1} \frac{\lambda(E\cap (r,1))}{1-r}=0,
\]
where $\lambda$ denotes Lebesgue measure. Each such set is of logarithmic density zero, see \cite[pp.\ 479--480]{FeSt09}, while it is easy to see that each set of finite logarithmic measure is of final density zero.

They also consider a positive increasing and piecewise continuously differentiable function $\varphi$ on some $[a,\infty)$, $a\geq 0$, so that
\[
\int_a^\infty \frac{1}{y\varphi(y)}\intd y<\infty \ \text{ and } \ \lim_{y\to\infty}\frac{1}{\varphi(y)}\varphi\Big(\frac{y}{\varphi(y)}\Big) =1.
\]

\begin{theorem}[Fenton, Strumia]\label{t-fest}
Let $f$ be an analytic function on the unit disk with unbounded sequence of Taylor coefficients. Then there is a set $E\subset [0,1)$ of final density zero and some $C>0$ such that
\[
M_f(r)\leq C \frac{\mu_f(r)}{1-r}\sqrt{\log \mu_f(r)} \varphi\Big(\frac{\log \mu_f(r)}{1-r}\Big), \ r\notin E.
\]
\end{theorem}

The hypothesis that the sequence of Taylor coefficients is unbounded leads to a simplification of the result in \cite{FeSt09} and only excludes the uninteresting case when $M_f(r)\leq C \frac{1}{1-r}$ and hence $M_f(r)\leq C \frac{\mu_f(r)}{1-r}$ for all $r$.

For the function $\varphi(y)=(\log y)^{1+\delta}$, $\delta>0$, one obtains the inequality
\[
M_f(r)\leq C \frac{\mu_f(r)}{1-r}\Big(\log\frac{1}{1-r}\Big)^{1+\delta}\big(\log \mu_f(r)\big)^\frac{1}{2} \big(\log\log \mu_f(r)\big)^{1+\delta}, \ r\notin E.
\]

More generally, if $\varphi(y)=\log y \log_2y \cdots \log_{n-2} y (\log_{n-1} y)^{1+\delta}$, $n\geq 2$, $\delta>0$, then one arrives at
\[
M_f(r)\leq C \frac{\mu_f(r)}{1-r}\Big(\log\frac{1}{1-r}\Big)^{1+\delta}\big(\log \mu_f(r)\big)^\frac{1}{2} \log_2 \mu_f(r)\cdots \log_{n-1} \mu_f(r)\big(\log_n \mu_f(r)\big)^{1+\delta}
\]
for $r\notin E$. 

Motivated presumably by Sule\u{\i}\-ma\-nov's paper, Skaskiv, Kuryliak and their collaborators have been studying the Wiman-Valiron inequality and its variants intensively over the last 25 years; let us refer to the recent papers by Skaskiv and Kuryliak \cite{KuSk23}, \cite{SkKu20} and the literature cited there. Among other things, they study simultaneously the case of the plane and the disk, and they measure exceptional sets in the sense of $h$-logarithmic measure.

\begin{definition}\label{d-hlog}
 Let $0<R\leq \infty$. Let $h$ be a positive increasing function on $[\rho,R)$ with $\int_{\rho}^R\frac{h(r)}{r}\intd r=\infty$ for some $\rho\in [0,R)$. Then a set $E\subset [0,R)$ is said to be of \emph{finite $h$-logarithmic measure} if 
\[
\int_{E\cap [\rho,R)} \frac{h(r)}{r}\intd r<\infty.
\]
\end{definition}

For $h(r)=1$, $r\in [1,\infty)$, and $h(r)=\frac{1}{1-r}$, $r\in [0,1)$, this reduces to the notion of finite logarithmic measure on $[0,\infty)$ and on $[0,1)$, respectively (note that the factor $\frac{1}{r}$ only plays a role when $R=\infty$). 

Here is the main result of \cite{SkKu20}, where $h$ is a function as in the above definition and $\psi_1, \psi_2$ are positive increasing functions on some interval $[a,\infty)$, $a>0$, with
\[
\int_a^\infty\frac{1}{\psi_k(y)}\intd y<\infty, \ k=1,2.
\]

\begin{theorem}[Skaskiv, Kuryliak]\label{t-SK}
Let $0<R\leq \infty$. Let $f$ be an unbounded analytic function on $\{z\in\mathbb{C}:|z|<R\}$. Then there is a set $E\subset [0,R)$ of finite $h$-logarithmic measure and some $C>0$ such that
\[
M_f(r)\leq C \mu_f(r)\sqrt{h(r)\psi_2\big(h(r)\psi_1\big(\log (h(r) \mu_f(r))\big)\big)}, \ r\notin E.
\]
\end{theorem}

Skaskiv and Kuryliak consider, in particular, the case when $\psi_1(y)=\psi_2(y)=y(\log y)^{1+\delta}$ with $\delta >0$. If $R=1$ and $h(r)=\frac{1}{1-r}$, this would give them the inequality
\begin{equation}\label{eq-SK}
M_f(r)\leq C \frac{\mu_f(r)}{1-r}\Big(\log \frac{1}{1-r}\Big)^{\frac{1}{2}+\delta}\Big(\log \frac{\mu_f(r)}{1-r}\Big)^{\frac{1}{2}}\Big(\log\log \frac{\mu_f(r)}{1-r}\Big)^{1+\delta}, \ r\notin E
\end{equation}
for some set $E\subset [0,1)$ of finite logarithmic measure, see also \cite[Theorem 1]{SkKu20}. This improves on Sule\u{\i}manov's inequality \eqref{eq-sul}.

More generally, for $\psi_1(y)=\psi_2(y)=y\log y\log_2y\cdots(\log_{n-1} y)^{1+\delta}$, $n\geq 2$, $\delta >0$, one obtains that
\begin{equation}\label{eq-SKbis}
M_f(r)\leq C \frac{\mu_f(r)}{1-r}\Big(\log \frac{1}{1-r}\Big)^{\frac{1}{2}+\delta}\Big(\log \frac{\mu_f(r)}{1-r}\Big)^{\frac{1}{2}}\log_2\frac{\mu_f(r)}{1-r}\cdots\log_{n-1} \frac{\mu_f(r)}{1-r}\Big(\log_n \frac{\mu_f(r)}{1-r}\Big)^{1+\delta}
\end{equation}
for $r\notin E$, which improves \eqref{eq-sul2}.

\begin{remark}
All the particular Wiman-Valiron inequalities on the unit disk stated above are close to optimal. If $f$ is given by $f(z)=e^{\frac{1}{(1-z)^\rho}}$, $\rho>0$ (K\"ov\'ari \cite{Kov66}) or $f(z)=\sum_{n=1}^\infty e^{n^\varepsilon}z^n$, $0<\varepsilon<1$ (Sule\u{\i}manov \cite{Sul80}) then
\begin{equation}\label{eq-wvDopt}
M_f(r) \geq C \frac{\mu_f(r)}{1-r}\Big(\log \frac{\mu_f(r)}{1-r}\Big)^{\frac{1}{2}}
\end{equation}
for some constant $C>0$ and all sufficiently large $r < 1$.
\end{remark}

The cacophony of Wiman-Valiron inequalities on the unit disk is unfortunate: there is no consensus on the right place for the factor $\frac{1}{1-r}$ and on the right notion of exceptional set. Contrast this with inequality \eqref{eq-wvC}, which seems to be considered as the canonical Wiman-Valiron inequality on the plane.

\section{A general Rosenbloom-K\"ov\'ari-type theorem}

We now want to derive, in the spirit of Skaskiv and Kuryliak, a general Rosenbloom-K\"ov\'ari-type theorem that holds on any disk in $\mathbb{C}$ and in which the exceptional set is of finite $h$-logarithmic measure; but see Remark \ref{r-comp}(a) below. We will moreover show that this result implies the theorem of Skaskiv and Kuryliak. 

Throughout this section, let $0<R\leq\infty$. We consider the same setting as above: let $h$ be a positive increasing function on $[\rho,R)$ with
\[
\int_{\rho}^R\frac{h(r)}{r}\intd r=\infty
\]
for some $\rho\in [0,R)$, and let $\psi_1, \psi_2$ be positive increasing functions on some interval $[a,\infty)$, $a>0$, with
\[
\int_a^\infty\frac{1}{\psi_k(y)}\intd y<\infty, \ k=1,2.
\]

\begin{theorem}\label{t-main}
Let $0<R\leq\infty$. Let $f$ be an unbounded analytic function on $\{z\in\mathbb{C}:|z|<R\}$. Then there is a set $E\subset [0,R)$ of finite $h$-logarithmic measure and some $C>0$ such that
\[
M_f(r)\leq C \mu_f(r)\sqrt{h(r)\psi_2\big(h(r)\psi_1\big(\log M_f(r)\big)\big)}, \ r\notin E.
\]
\end{theorem}

Maybe surprisingly, given the long history of the inequality, the proof is rather short, so that we can provide it here with full details.

The first ingredient is the outcome of Rosenbloom's clever probabilistic approach; his argument is more transparent when working in the image of the interval $(0,R)$ under the mapping $x=\log r$. 

\begin{lemma}[Rosenbloom]\label{l-rosen}
Let $f(z)=\sum_{n=0}^\infty a_n z^n$ be an analytic function on $\{z\in\mathbb{C}:|z|<R\}$ that is not a monomial and with $a_n\geq 0$ for $n\geq 0$. For $x\in (-\infty,\log R)$, set $F(x)= f(e^{x})$ and $g(x)=\log F(x)$. Then $g''(x)>0$, and there is a constant $C>0$ such that, for all $x\in (-\infty,\log R)$,
\[
F(x)\leq C\mu_f(e^x)\sqrt{g''(x)}.
\]
\end{lemma}

\begin{proof} We fix $x\in (-\infty,\log R)$. Then let $X$ be an $\mathbb{N}_0$-valued random variable so that
\[
\mathbb{P}(X=n)= \frac{a_ne^{nx}}{F(x)},\ n\geq 0;
\]
note that $F(x)>0$. A simple calculation shows that $X$ has expectation $\mathbb{E}(X)=g'(x)$ and variance $\text{Var}(X)=g''(x)$, the latter being strictly positive since $f$ is not a monomial. It then follows from Chebyshev's inequality that
\[
\mathbb{P}\big(|X-g'(x)|< c\sqrt{g''(x)}\big) \geq 1-c^{-2}
\]
for any $c>1$. Evaluating the left-hand probability we find that
\[
F(x) \leq \frac{1}{1-c^{-2}} \sum_{|n-g'(x)|<c\sqrt{g''(x)}} a_n e^{nx}\leq \frac{2c}{1-c^{-2}} \mu_f(e^x) \sqrt{g''(x)},
\]
which implies the claim.
\end{proof}

The next lemma is standard in our context.

\begin{lemma}\label{l-standard}
Let $h$ be a function as above and $\psi$ be a positive increasing function on some interval $[a,\infty)$, $a>0$, with $\int_a^\infty\frac{1}{\psi(y)}\emph{\intd} y<\infty$. Let $g$ be an unbounded increasing continuously differentiable function on some interval $[b,\log R)$, $b<\log R$. Then there is a set $E\subset [0,R)$ of finite $h$-logarithmic measure such that
\[
g'(x) < h(e^x)\psi\big(g(x)\big), \ e^x\notin E.
\]
\end{lemma}

\begin{proof} Let $x_0\in [b,\log R)$ be such that $g(x_0)> a$, and let $\rho$ be as in the hypothesis on $h$. Let
\[
\widetilde{F}= \{ x\in [x_0,\log R) : e^x \geq \rho, g'(x) \geq h(e^x) \psi\big(g(x)\big) \}
\]
and $F=e^{\widetilde{F}}$. Then we have that 
\[
\int_{\widetilde{F}} h(e^x)\intd x \leq \int_{\widetilde{F}} \frac{g'(x)}{\psi\big(g(x)\big)}\intd x \leq \int_{x_0}^{\log R} \frac{g'(x)}{\psi\big(g(x)\big)}\intd x.
\]
Performing the changes of variables $r=e^x$ on the left-hand side and $y=g(x)$ on the right-hand side (note that $g$ need not be strictly increasing, see \cite[p.\ 156]{Rud87}) we obtain that
\[
\int_{F} \frac{h(r)}{r}\intd r \leq \int_a^\infty \frac{1}{\psi(y)}\intd y<\infty.
\]
Thus the set
\[
E= \big[0, \max(\rho,e^{x_0})\big) \cup F
\]
has finite $h$-logarithmic measure, and for any $r=e^x\notin E$ we have that $g'(x) < h(e^x) \psi\big(g(x)\big)$.
\end{proof}

We are now ready to prove the main theorem.

\begin{proof}[Proof of Theorem \ref{t-main}] 
Let $f(z)=\sum_{n=0}^\infty a_nz^n$, $|z|<R$. If we replace $a_n$ by $|a_n|$, $n\geq 0$, we do not change $\mu_f$ and we do not decrease $M_f$. Also, if $f$ is a monomial, then the claim is trivially true. So we can assume that the Taylor coefficients of $f$ are positive and that $f$ is not a monomial. 

As before, set $F(x)= f(e^{x})$ and $g(x)=\log F(x)$ for $x \in (-\infty,\log R)$. By Lemma \ref{l-rosen} we have that $g''(x)\geq 0$, so that $g$ is a convex increasing function which, by assumption, tends to infinity as $x\to\log R$. Thus also $g'$ tends to infinity. We can therefore apply Lemma \ref{l-standard} to $g'$ and $\psi_2$, and we can apply it to $g$ itself with $\psi_1$. Let $E_2$ and $E_1$ be the corresponding exceptional sets, and let $E=E_2\cup E_1$. Then $E$ has finite $h$-logarithmic measure, and we have that
\[
g''(x) \leq h(e^x)\psi_2\big(g'(x)\big) \leq h(e^x)\psi_2\big(h(e^x)\psi_1\big(g(x)\big)\big),\ e^x\notin E,
\]
where we have also used that $\psi_2$ is increasing. Combining this with Lemma \ref{l-rosen} we have that
\[
F(x)\leq C\mu_f(e^x)\sqrt{h(e^x)\psi_2\big(h(e^x)\psi_1\big(g(x)\big)\big)}, \ e^x\notin E,
\]
with some constant $C>0$. It remains to replace $e^x$ by $r$ to complete the proof.
\end{proof}

Let us discuss the result.

\begin{remark}\label{r-comp}
(a) To the best of our knowledge, Theorem \ref{t-main} is a new result. In hindsight, however, Köv\'ari could rightly have claimed it as his own. Indeed, his aim in \cite{Kov66} was to obtain a majorization by $\sqrt{\varphi(\log M_f(r))}= \sqrt{\psi_2(\psi_1(\log M_f(r)))}$ on the unit disk as Rosenbloom had done on the plane. This forced him to add conditions on the functions $\psi_1$ and $\psi_2$, and his exceptional set $E$ became rather large, see \cite[p.\ 134]{Kov66}. Had he contented himself with a majorization by $\sqrt{\frac{1}{1-r}\psi_2(\frac{1}{1-r}\psi_1(\log M_f(r)))}$, his exceptional sets $E_1$ and $E_2$ on the same page 134 would have given him our result for $h(r)=\frac{1}{1-r}$ on the unit disk. Moreover, on page 131 he mentions that his result would remain true for a general function $h$. 

(b) Theorem \ref{t-main} contains Rosenbloom's Theorem \ref{t-rosen} by taking $h(r)=1$ for $r\geq 0$.

(c) Theorem \ref{t-main} implies Theorem \ref{t-SK} by Skaskiv and Kuryliak. Indeed, taking $\psi_1(y)=e^{y/2}$ and $\psi_2(y)=y^2$, $y>0$, one obtains the inequality $M_f(r)\leq C_1\mu_f(r) h(r)^{\frac{3}{2}}\sqrt{M_f(r)}$ outside a set $E$ of finite $h$-logarithmic measure, and thus $M_f(r)\leq C_1^2\mu_f(r)^2 h(r)^3$, so that $\log M_f(r) \leq C_2\log(\mu_f(r)h(r))$ for some constant $C_2>0$ and for $r\notin E$ large. Then Theorem \ref{t-SK} follows for the $\psi_1$ and $\psi_2$ given there by applying Theorem \ref{t-main} to $y\to\psi_1(\frac{1}{C_2}y)$ and $\psi_2$.
\end{remark}

We end with the example that has motivated this paper.

\begin{example}
We consider an unbounded analytic function $f$ on the unit disk. Then inequality \eqref{eq-SK} and the fact that $1\leq \alpha\mu_f(r)$ for some $\alpha>0$ and all large $r<1$ gives us the following: for any $\delta>0$, there is a set $E\subset [0,1)$ of finite logarithmic measure such that
\[
M_f(r)\leq C \frac{\mu_f(r)}{1-r}\Big(\log \frac{1}{1-r}\Big)^{\frac{1}{2}}\Big(\log \frac{\mu_f(r)}{1-r}\Big)^{\frac{1}{2}+\delta}, \ r\notin E.
\]
Here, as in other inequalities in this paper concerning the unit disk, the factor $(\log \frac{1}{1-r})^{\frac{1}{2}}$ seems out of place. And indeed, we have seen in \eqref{eq-wvD} that
\[
M_f(r)\leq C \frac{\mu_f(r)}{1-r}\Big(\log \frac{\mu_f(r)}{1-r}\Big)^{\frac{1}{2}+\delta}, \ r\notin E,
\]
where, however, $E$ is a (possibly larger) set of logarithmic density zero. But we can do better. Under the usual assumptions on the function $h$, and arguing as for inequality \eqref{eq-SKbis}, one obtains that, for any $n\geq 2$ and any $\delta>0$, there is a set $E\subset [0,1)$ of finite $h$-logarithmic measure such that
\begin{equation}
\begin{split}\label{eq-SK4}
M_f(r)\leq C h(r)\mu_f(r)\big(\log h(r)\big)^{\frac{1}{2}+\delta}&\big(\log (h(r)\mu_f(r))\big)^{\frac{1}{2}}\log_2(h(r)\mu_f(r))\cdots\\
&\cdots\log_{n-1} (h(r)\mu_f(r))\big(\log_n (h(r)\mu_f(r))\big)^{1+\delta}, \ r\notin E.
\end{split}
\end{equation}
Note that it suffices to apply Theorem \ref{t-SK} instead of the stronger Theorem \ref{t-main}.

Taking, in particular, 
\[
h(r)=\frac{1}{(1-r)\log\frac{1}{1-r}}, \ r\in [1-1/e,1),
\]
then inequality \eqref{eq-SK4} gives us, for $\delta\leq \frac{1}{2}$, the estimate
\begin{align}
M_f(r)&\leq C \frac{\mu_f(r)}{1-r}\Big(\log\frac{1}{1-r}\Big)^{\delta-\frac{1}{2}}\Big(\log \frac{\mu_f(r)}{1-r}\Big)^{\frac{1}{2}}\log_2\frac{\mu_f(r)}{1-r}\cdots\log_{n-1} \frac{\mu_f(r)}{1-r}\Big(\log_n \frac{\mu_f(r)}{1-r}\Big)^{1+\delta}\notag\\
&\leq C \frac{\mu_f(r)}{1-r}\Big(\log \frac{\mu_f(r)}{1-r}\Big)^{\frac{1}{2}}\log_2\frac{\mu_f(r)}{1-r}\cdots\log_{n-1} \frac{\mu_f(r)}{1-r}\Big(\log_n \frac{\mu_f(r)}{1-r}\Big)^{1+\delta}\label{eq-log}
\end{align}
for all $r$ outside a set $E$ of finite $\frac{1}{(1-r)\log\frac{1}{1-r}}$-logarithmic measure. And it is not difficult to see that every such set is of logarithmic density zero. Thus, inequality \eqref{eq-log} improves inequality \eqref{eq-wvD2} and thus also \eqref{eq-wvD}.

On the other hand, a set of finite $\frac{1}{(1-r)\log\frac{1}{1-r}}$-logarithmic measure is not necessarily of finite logarithmic measure.
\end{example}

The inequalities \eqref{eq-wvC} and \eqref{eq-log} suggest the following.

\begin{question}
Let $f$ be an unbounded analytic function on the unit disk.

(a) For any $n\geq 2$ and $\delta>0$, does there exist a constant $C>0$ and a set $E$ of finite logarithmic measure so that
\[
M_f(r)\leq C \frac{\mu_f(r)}{1-r}\Big(\log \frac{\mu_f(r)}{1-r}\Big)^{\frac{1}{2}}\log_2 \frac{\mu_f(r)}{1-r}\cdots\log_{n-1} \frac{\mu_f(r)}{1-r}\Big(\log_n \frac{\mu_f(r)}{1-r}\Big)^{1+\delta}, \ r\notin E
\]
holds?

(b) In particular,  for any $\delta>0$, does there exist a constant $C>0$ and a set $E$ of finite logarithmic measure so that
\[
M_f(r)\leq C \frac{\mu_f(r)}{1-r}\Big(\log \frac{\mu_f(r)}{1-r}\Big)^{\frac{1}{2}+\delta}, \ r\notin E
\]
holds?
\end{question}

These inequalities would be almost optimal, see \eqref{eq-wvDopt}.

For an application of Wiman-Valiron inequalities to random entire functions and, subsequently, to linear dynamics we refer to the forthcoming paper \cite{AgGE23}.


\begin{thebibliography}{99}
\bibitem{AgGE23} K. Agneessens and K.-G. Grosse-Erdmann, Rate of growth of random analytic functions, with an application to linear dynamics, preprint.

\bibitem{FeSt09} P. C. Fenton and M. M. Strumia, Wiman-Valiron theory in the disc, \emph{J. Lond. Math. Soc. (2)} 79 (2009), 478--496.

\bibitem{Hay74} W.~K. Hayman, The local growth of power series: a survey of the {W}iman-{V}aliron method, \emph{Canad. Math. Bull.} 17 (1974), 317--358.

\bibitem{Hay89} W.~K. Hayman, \emph{Subharmonic functions. {V}ol. 2}, Academic Press, Inc., London 1989.

\bibitem{Kov66} T. Kovari, On the maximum modulus and maximum term of functions analytic in the unit disc, \emph{J. London Math. Soc.} 41 (1966), 129--137.

\bibitem{KuSk23} A. O. Kuryliak and O. B. Skaskiv, Sub-Gaussian random variables and Wiman's inequality for analytic functions, \textit{Carpathian Math. Publ.} 15 (2023), 306--314. 

\bibitem{Ros62} P.~C. Rosenbloom, Probability and entire functions, in: \emph{Studies in mathematical analysis and related topics}, 325--332, Stanford Univ. Press, Stanford, CA 1962.

\bibitem{Rud87} W. Rudin, \emph{Real and complex analysis}, third edition, McGraw-Hill Book Co., New York 1987.

\bibitem{SkKu20} O. B. Skaskiv and A. O. Kuryliak, Wiman's type inequality for analytic and entire functions and $h$-measure of an exceptional sets, \emph{Carpathian Math. Publ.} 12 (2020), 492--498.

\bibitem{Sul80} N. M. Sule\u{\i}manov, Wiman-Valiron-type estimates for power series with a finite radius of convergence and their exactness (Russian), \emph{Dokl. Akad. Nauk SSSR} 253 (1980), 822--824. English translation in: \textit{Soviet Math. Dokl.} 22 (1980), 190--192. 

\bibitem{Val18}
G. Valiron, Sur le maximum du module des fonctions entières, \emph{C. R. Acad. Sci. Paris} 166 (1918), 605--608.

\bibitem{Val20}
G. Valiron, Les th\'eor\`emes g\'en\'eraux de M. Borel dans la th\'eorie des fonctions enti\`eres, \emph{Ann. Sci. \'Ecole Norm. Sup. (3)} 37 (1920), 219--253.

\bibitem{Val23}
G. Valiron, \emph{Lectures on the general theory of integral functions}, Deighton, Bell and Co., Cambridge 1923. 

\bibitem{Wim14}
A. Wiman, \"Uber den Zusammenhang zwischen dem Maximalbetrage einer analytischen Funktion und dem gr\"ossten Gliede der zugeh\"origen Taylor'schen Reihe, \emph{Acta Math.} 37 (1914), 305--326.
\end{thebibliography}
\end{document}